\documentclass[12pt,a4paper]{amsart}

\usepackage{amsmath, amsthm, amssymb,latexsym}
\usepackage{enumerate}

\theoremstyle{plain}
\newtheorem{theorem}{Theorem}[section]
\newtheorem{lemma}[theorem]{Lemma}
\newtheorem{mainth}{Theorem}
\newtheorem{arith}{Theorem}

\newtheorem{proposition}[theorem]{Proposition}
\theoremstyle{definition}
\newtheorem{definition}[theorem]{Definition}

\newtheorem*{question}{Problem}
\theoremstyle{remark}
\newtheorem{remark}[theorem]{Remark}
\newtheorem{fact}[theorem]{Fact}
\numberwithin{equation}{section}

\newcommand{\bC}{\mathbb{C}}
\newcommand{\bN}{\mathbb{N}}
\newcommand{\bP}{\mathbb{P}}
\newcommand{\bQ}{\mathbb{Q}} 
\newcommand{\bR}{\mathbb{R}}
\newcommand{\bZ}{\mathbb{Z}}  
\newcommand{\sP}{\mathsf{P}}
\newcommand{\sD}{\mathsf{D}}
\newcommand{\sF}{\mathsf{F}}
\newcommand{\sJ}{\mathsf{J}}
\newcommand{\sK}{\mathsf{K}}
\newcommand{\sA}{\mathsf{A}}
\newcommand{\cB}{\mathcal{B}}
\newcommand{\cS}{\mathcal{S}}
\newcommand{\cE}{\mathcal{E}} 
\newcommand{\can}{\operatorname{can}}
\newcommand{\supp}{\operatorname{supp}}
\newcommand{\diam}{\operatorname{diam}}
\newcommand{\Int}{\operatorname{int}}

\newcommand{\rd}{\mathrm{d}}
\newcommand{\Id}{\mathrm{Id}}

\begin{document} 

\title[Equidistribution in non-archimedean dynamics]{
Adelic equidistribution, 
Characterization of equidistribution, and
a general equidistribution theorem in non-archimedean dynamics}

\author[Y\^usuke Okuyama]{Y\^usuke Okuyama}
\address{
Division of Mathematics,
Kyoto Institute of Technology,
Sakyo-ku, Kyoto 606-8585 Japan}
\email{okuyama@kit.ac.jp}

\date{\today}

\begin{abstract}
We determine when the equidistribution property for possibly moving targets holds
for a rational function of degree more than one on the projective line
over an algebraically closed field of any characteristic and
complete with respect to a non-trivial
absolute value. This characterization could be useful in the positive characteristic case.
Based on the variational argument,
we give a purely local proof of the adelic equidistribution theorem for 
possibly moving targets,
which is due to Favre and Rivera-Letelier,
using a dynamical Diophantine approximation theorem 
by Silverman and by Szpiro--Tucker.
We also give a proof of a general equidistribution theorem for possibly moving targets,
which is due to Lyubich in the archimedean case and 
due to Favre and Rivera-Letelier
for constant targets in the non-archimedean and any characteristic case and
for moving targets in the non-archimedean and $0$ characteristic case.
\end{abstract}



\maketitle

\section{Introduction}\label{sec:intro}

Let $K$ be an algebraically closed field of any characteristic and complete with respect
to a non-trivial and possibly non-archimedean absolute value $|\cdot|$, and 
let $f\in K(z)$ be a rational function of degree $d>1$ on 
the projective line $\bP^1=\bP^1(K)$ over $K$.
The Berkovich projective line $\sP^1=\sP^1(K)$ over $K$ 
provides a compactification of the classical $\bP^1$, 
containing $\bP^1$ as a dense subset. 
Under the assumption that $K$ is algebraically closed, 
$K$ is archimedean if and only if $K\cong\bC$, and
then $\sP^1(\bC)\cong\bP^1(\bC)$. 
The action of $f$ on $\bP^1$ canonically
extends to a continuous, open, surjective
and fiber-discrete endomorphism on $\sP^1$, preserving $\sP^1$ and $\sP^1\setminus\bP^1$.
The exceptional set of (the extended) $f$ is 
\begin{gather*}
 E(f):=\left\{a\in\bP^1:\#\bigcup_{n\in\bN}f^{-n}(a)<\infty\right\},
\end{gather*}
which agrees with the set of all superattracting periodic points $a\in\bP^1$ of $f$
such that $\deg_{f^j(a)} f=d$ for any $j\in\bN$.
The {\itshape Berkovich} Julia set of $f$ is
 \begin{gather*}
 \sJ(f):=\left\{\cS\in\sP^1:
 \bigcap_{U:\text{open in }\sP^1, \cS\in U}\left(\bigcup_{n\in\bN}f^n(U)\right)=\sP^1\setminus
 E(f)\right\}
 \end{gather*}
(cf.\ \cite[Definition 2.8]{FR09}).
Let $\delta_\cS$ be the Dirac measure on $\sP^1$ at a point $\cS\in\sP^1$. 
For each rational function $a\in K(z)$, which we will call a {\itshape possibly moving target},
on $\bP^1$ and each $n\in\bN$, 
let us consider the probability Radon measure
\begin{gather}
 \nu_n^a=\nu_{f^n}^a:=\frac{1}{d^n+\deg a}\sum_{w\in\bP^1:f^n(w)=a(w)}\delta_w\label{eq:roots}
\end{gather}
on $\sP^1$. Here the sum takes into account the (algebraic) multiplicity of each root of the
equation $f^n(\cdot)=a(\cdot)$ in $\bP^1$. 
In Section \ref{sec:facts}, among other generalities,
we recall a variational characterization of the equilibrium (or canonical) measure $\mu_f$
of $f$ on $\sP^1$
as a unique solution of a Gauss variational problem.

Our principal result determines the conditions on $f$ and $a$ under which 
the equidistribution property
\begin{gather}
 \lim_{n\to\infty}\nu_n^a=\mu_f\quad\text{weakly on }\sP^1\label{eq:general}
\end{gather}
holds. Let us denote the normalized chordal distance on $\bP^1$ by $[z,w]$.

\begin{mainth}\label{th:julia}
 Let $K$ be an algebraically closed field of any characteristic
 and complete with respect to a non-trivial absolute value.
 Let $f\in K(z)$ be a rational function on $\bP^1$ of degree $d>1$ and
 let $a\in K(z)$ be a rational function on $\bP^1$. 
Then for every sequence 
$(n_j)\subset\bN$ tending to $\infty$, the following three conditions are
equivalent$:$ 
\begin{enumerate}
 \item The equidistribution property 
\begin{gather}
 \lim_{j\to\infty}\nu_{n_j}^a=\mu_f\quad\text{on }\sP^1\label{eq:subgeneral}
\end{gather}
       holds. Equivalently, for each weak limit $\nu$ of a subsequence of $(\nu_{n_j}^a)$, 
\begin{gather}
 \nu=\mu_f;\tag{\ref{eq:subgeneral}'}\label{eq:weaklim}
\end{gather}
 \item each weak limit $\nu$ of a subsequence of $(\nu_{n_j}^a)$ satisfies
 \begin{gather}
  \supp\nu\subset\sJ(f);\label{eq:support}
 \end{gather}
 \item under the additional assumption that $K$ is non-archimedean, 
       on $\sP^1\setminus\bP^1$, we have
 \begin{gather}
  \lim_{j\to\infty}\frac{1}{d^{n_j}}\log[f^{n_j},a]_{\can}(\cdot)=0.\label{eq:hyperbolic}
 \end{gather}
\end{enumerate}
Under these three conditions, we have
 \begin{gather}
  \lim_{j\to\infty}\frac{1}{d^{n_j}}\int_{\sP^1}\log[f^{n_j},a]_{\can}(\cdot)\rd\mu_f=0.\label{eq:Valiron}
 \end{gather}
 Moreover, if $a$ is constant, then \eqref{eq:Valiron} holds
 without assuming \eqref{eq:subgeneral}, \eqref{eq:support} or \eqref{eq:hyperbolic}.

 Here, the proximity function $\cS\mapsto[f^n,a]_{\can}(\cS)$ of $f^n$
 $(n\in\bN)$ and $a$
 on $\sP^1$ is the unique continuous extension of $z\mapsto [f^n(z),a(z)]$ on $\bP^1$ to $\sP^1$. 
 For its construction, see Proposition $\ref{th:extension}$.
\end{mainth} 

In Section \ref{sec:equidist}, we show Theorem \ref{th:julia}
based on the above variational characterization of $\mu_f$.
Theorem \ref{th:julia} is partly motivated by the following
dynamical Diophantine approximation result.
For a number field $k$ with a non-trivial absolute value (or place) $v$,
set $K=\bC_v$ with the extended $v$
(e.g., $K=\bC_p$ for $k=\bQ$ with $p$-adic norm $v$)
and assume that $f\in k(z)$, i.e., that $f$ has its coefficients in $k$. Then
the dynamical Diophantine approximation theorem due to Silverman \cite[Theorem E]{Silverman93} and
Szpiro--Tucker \cite[Proposition 5.3 (in the preprint version, Proposition 4.3)]{ST05}
asserts that for every constant $a\in\bP^1(\overline{k})\setminus E(f)$ and 
every $z\in\bP^1(\overline{k})$ which is wandering under $f$, i.e.,
$\#\{f^n(z):n\in\bN\}=\infty$, we have
\begin{gather}
 \lim_{n\to\infty}\frac{1}{d^n}\log[f^n(z),a]_v=0.\label{eq:diophantine}
\end{gather}
Here $\overline{k}$ denotes the algebraic closure of $k$, and
the notation $[z,w]_v$ emphasizes the dependence of $[z,w]$ on $v$. 
Theorem \ref{th:julia} gives a partial generalization \eqref{eq:diophantine} to general $K$
for possibly non-constant $a$.

In Section \ref{sec:general}, based on the variational argument
and (\ref{eq:diophantine}),
we give a purely local proof of the following adelic equidistribution theorem for 
possibly moving targets, which is a special case of Favre and Rivera-Letelier 
\cite[Th\'eor\`emes A et B]{FR09} (Theorems \ref{th:general} and \ref{th:moving} below)
for non-archimedean $K$ of characteristic $0$.

\begin{arith}\label{th:FR}
 Let $k$ be a number field with a non-trivial absolute value $v$, and
 let $f\in k(z)$ be a rational function on $\bP^1(\bC_v)$ of degree $d>1$ whose
 coefficients are in $k$.
 Then for every rational function $a\in\overline{k}(z)$ on $\bP^1(\bC_v)$
 which is not identically equal to a value in $E(f)$ and whose coefficients are
 in $\overline{k}$,
 $\lim_{n\to\infty}\nu_n^a=\mu_{f,v}$
 weakly on $\sP^1(\bC_v)$.  
Here the notation $\mu_{f,v}$ emphasizes the dependence of $\mu_f$ on $v$.
\end{arith}

 For another application 
 (quantitative equidistribution for non-exceptional algebraic constants)
 of the dynamical Diophantine approximation \eqref{eq:diophantine}
 to adelic dynamics, see \cite{OkuFekete}.

For general $K$, the equidistribution theorem for {\itshape constant} $a\in\bP^1\setminus E(f)$
is due to Brolin \cite{Brolin}, Lyubich \cite{Lyubich83}, Freire, Lopes and
Ma\~n\'e \cite{FLM83} for archimedean $K$ and due to Favre and Rivera-Letelier 
\cite[Th\'eor\`eme A]{FR09} for non-archimedean $K$.

\begin{theorem}\label{th:general}
 Let $K$ be an algebraically closed field of any characteristic
 and complete with respect
 to a non-trivial absolute value.
 Let $f\in K(z)$ be a rational function on $\bP^1$ of degree $d>1$, and
 $a\in K(z)$ be a constant function. Then
 $\lim_{n\to\infty}\nu_n^a=\mu_f$ weakly on $\sP^1$ if and only if
\begin{gather}
 a\in\bP^1(K)\setminus E(f).\label{eq:nonexceptional}
\end{gather}
\end{theorem}

In Section \ref{sec:general}, 
we give a proof of Theorem \ref{th:general},
the fundamental equivalence between \eqref{eq:general} and \eqref{eq:nonexceptional}
for constant $a$,
based on the variational argument and on the classification of cyclic Berkovich Fatou components
of $f$ (see Theorem \ref{th:classification}).

For general $K$ of characteristic $0$, 
the equidistribution theorem for {\itshape moving} targets
is due to Lyubich \cite[Theorem 3]{Lyubich83} (see also Tortrat \cite[\S IV]{Tortrat87}) 
for archimedean $K$ 
and due to Favre and Rivera-Letelier \cite[Th\'eor\`eme B]{FR09} for non-archimedean $K$
of characteristic $0$.  
 
\begin{theorem}\label{th:moving}
 Let $K$ be an algebraically closed field of characteristic $0$
 and complete with respect
 to a non-trivial absolute value.
 Let $f\in K(z)$ be a rational function on $\bP^1$ of degree $d>1$. Then 
 for every non-constant rational function $a\in K(z)$ on $\bP^1$, 
 $\lim_{n\to\infty}\nu_n^a=\mu_f$ weakly on $\sP^1$.
\end{theorem}

In Section \ref{sec:general}, we also describe
how the variational argument together with the dynamical
uniformization on the quasiperiodicity domain 
$\cE_f$ (see Theorem \ref{th:uniformization})
yields Theorem \ref{th:moving}. This is foundational in our study of
the problem on the density of the classical repelling periodic points in 
the classical Julia set in non-archimedean dynamics \cite{OkuLog}.
Our proof of Theorem \ref{th:moving}
complements the original one given in \cite[\S3.4]{FR09}
(see also Remark \ref{th:incoincidence}).

In Section \ref{sec:example}, we discuss the case where $f$ and $a$ are polynomials,
and compute a concrete example.

We conclude this section with an open problem.

\begin{question}
Let $K$ be an algebraically closed field of positive characteristic and
complete with respect to a non-trivial absolute value.
Let $f\in K(z)$ be a rational function on $\bP^1$ of degree $d>1$.
Determine concretely all rational functions $a\in K(z)$ on $\bP^1$ 
which are {\itshape exceptional for} $f$ in that
the equidistribution \eqref{eq:general} does not hold.
\end{question}

We hope condition \eqref{eq:hyperbolic}
will be helpful for studying this problem.

\section{Background}\label{sec:facts}
For the foundations of potential theory on $\sP^1$, see 
\cite[\S 5 and \S 8]{BR10}, \cite[\S 7]{FJbook}, 
\cite[\S 1-\S 4]{Jonsson12}, \cite[Chapter III]{Tsuji59}. 
For a potential theoretic study of dynamics on $\sP^1$,
see \cite[\S 10]{BR10}, \cite[\S 3]{FR09}, \cite[\S 5]{Jonsson12}, \cite[Chapitre VIII]{BM01}. 
See also \cite{Benedetto10, Juan03} including non-archimedean dynamics.

Let $K$ be an algebraically closed field of any characteristic and complete with respect
to a non-trivial absolute value $|\cdot|$. Under the assumption that $K$ is algebraically closed,
$|K|:=\{|z|:z\in K\}$ is dense in $\bR_{\ge 0}$. We will say $K$ to be {\itshape non-archimedean} if
the strong triangle inequality $|z-w|\le\max\{|z|,|w|\}$ holds for all $z,w\in K$.
This in particular implies that the equality $|z-w|=\max\{|z|,|w|\}$ 
holds if $|z|\neq|w|$. When $K$ is non-archimedean, 
for every $a,b\in K$ and every $r\ge 0$, 
$\{z\in K:|z-a|\le r\}=\{z\in K:|z-b|\le r\}$
if $|b-a|\le r$, and the diameters of these sets with respect to $|\cdot|$ equal $r$.
If $K$ is not non-archimedean, then
$K$ is said to be archimedean. Under the assumption that $K$ is algebraically closed,
$K$ is archimedean if and only if $K\cong\bC$ as valued fields.

Let $\|\cdot\|$ be the maximum norm
on $K^2$ if $K$ is non-archimedean, and 
the Euclidean norm on $\bC^2$ if $K$ is archimedean $(\cong\bC)$. 
Put $p\wedge q:=p_0q_1-p_1q_0$ for $p=(p_0,p_1),q=(q_0,q_1)\in K^2$; 
let $\pi$ be the canonical projection $K^2\setminus\{0\}\to\bP^1=\bP^1(K)$,
and put $\infty:=\pi(0,1)$. The normalized chordal distance on $\bP^1$ is 
\begin{gather*}
 [z,w]:=\frac{|p\wedge q|}{\|p\|\cdot\|q\|}\in[0,1],
\end{gather*}
where $p\in\pi^{-1}(z),q\in\pi^{-1}(w)$. 
We usually 
identify $K$ with $\bP^1\setminus\{\infty\}$ by the injection $z\mapsto \pi(1,z)$
on $K$.

For non-archimedean $K$, the Berkovich projective line $\sP^1 =\sP^1(K)$
is defined as an analytic space in the sense of Berkovich; 
see Berkovich's original monograph \cite{Berkovichbook}, 
as well as \cite[\S1, \S2]{BR10} for $\sP^1$. 
For archimedean $K$, we have $\sP^1 =\bP^1$.

\begin{fact}[Berkovich's classification of points in $\sP^1$]
 Suppose that $K$ is non-archimedean. A subset $\cB=\{z\in K:|z-a|\le r\}$ in $K$ 
 for some $a\in K$ and some $r=:\diam(\cB)\ge 0$
 is called a ($K$-closed) {\itshape disk}.
 Any two intersecting disks $\cB,\cB'$ satisfy either
 $\cB\subset\cB'$ or $\cB\supset\cB'$. 

 A point $\cS$ in the Berkovich projective line $\sP^1$ is either $\infty$ or
 is a cofinal class (or tail) of non-increasing and nested sequences of disks $(\cB_j)$. Here,
 two non-increasing and nested sequences of disks $(\cB_j),(\cB'_k)$ are cofinally equivalent 
 either if (i) $\bigcap_j\cB_j=\bigcap_k\cB'_k\neq\emptyset$ or 
 if (ii) $\bigcap_j\cB_j=\bigcap_k\cB'_k=\emptyset$,
 for any $j\in\bN$, $\cB_j$ contains $\cB'_N$ for some $N\in\bN$,
 and for any $k\in\bN$, $\cB'_k$ contains $\cB_{N'}$ for some $N'\in\bN$.
 A cofinal class of non-increasing and nested sequence of disks $(\cB_j)$ 
 is identified with the disk $\cB=\bigcap_{j\in\bN}\cB_j$ if it is non-empty.
 The projective line $\bP^1$ is regarded as the set of all
 disks $\cB$ with $\diam(\cB)=0$ and the point $\infty$ 
 (cf.\ \cite[\S 1]{BR10},  \cite[\S 6.1]{Benedetto10}, \cite[\S 2]{FR09}). 
\end{fact}

 Let $\Omega_{\can}$ be the Fubini-Study area element 
 on $\bP^1\cong\sP^1$ normalized as $\Omega_{\can}(\bP^1)=1$ for archimedean $K\cong\bC$,
 and the Dirac measure $\delta_{\cS_{\can}}$ on $\sP^1$ at 
 the Gauss (or canonical) 
 point $\cS_{\can}\in\sP^1$ determined by the disk $\{z\in K:|z|\le 1\}$ 
 for non-archimedean $K$. 

\begin{definition}[the generalized Hsia kernel]
 Suppose that $K$ is non-archimedean. 
 To each cofinal class $\cS$ of a non-increasing and nested 
 sequence of disks $(\cB_j)$, 
 set $\diam(\cS):=\lim_{j\to\infty}\diam(\cB_j)$.
 Then the function $\diam(\cdot)$ is continuous on $\sP^1\setminus\{\infty\}$.

 For cofinal classes $\cS,\cS'$ of non-increasing and nested sequences of disks 
 $(\cB_j),(\cB'_k)$, 
 respectively, let $\cS\wedge\cS'\in\sP^1$ be the smallest cofinal class 
 of non-increasing and nested sequence of disks $(\cB''_{\ell})$ such that
 for every $\ell\in\bN$, $\cB''_{\ell}$ contains $\cB_N\cup\cB'_{N'}$ for some $N,N'\in\bN$.
 Here a cofinal class of $(\cB''_{\ell})$ is said to be smaller than that of $(\cB'''_m)$ if
 for every $m\in\bN$, $\cB'''_m$ contains $\cB''_{N''}$ for some $N''\in\bN$.

For each $w\in\bP^1\setminus\{\infty\}$, 
 the function $|\cdot-w|:=\diam(\cdot\wedge w)$ on $\sP^1\setminus\{\infty\}$ 
 is a unique continuous extension of $|\cdot-w|$ on $\bP^1\setminus\{\infty\}$.
 We denote $|\cdot-0|$ by $|\cdot|$ in the case $w=0$. 

 The generalized Hsia kernel $[\cS,\cS']_{\can}$
 on $\sP^1$ with respect to the Gauss point $\cS_{\can}$
 is defined as
 \begin{gather*}
  [\cS,\cS']_{\can}:=\frac{\diam(\cS\wedge\cS')}{\max\{1,|\cS|\}\max\{1,|\cS'|\}}\in[0,1]
 \end{gather*} 
 for $\cS,\cS'\in\sP^1\setminus\{\infty\}$,
 $[\cS,\infty]_{\can}:=1/\max\{1,|\cS|\}$ for $\cS\in\sP^1\setminus\{\infty\}$,
 and $[\infty,\infty]_{\can}:=[\infty,\infty]=0$ 
 (see \cite[\S4]{BR10}, \cite[\S2.4]{FR09}). 

 By convention, for archimedean $K$, $[z,w]_{\can}$ is defined by $[z,w]$.
\end{definition}

 \begin{fact}\label{th:separately}
 The extension $[\cS,\cS']_{\can}$ is upper semicontinuous on $\sP^1\times\sP^1$,
 continuous nowhere in the diagonal of
 $(\sP^1\setminus\bP^1)\times(\sP^1\setminus\bP^1)$
 (indeed, $[\cS,\cS]_{\can}$ is continuous nowhere on $\sP^1\setminus\bP^1$), but
 continuous elsewhere on $\sP^1\times\sP^1$. On the other hand, $[\cS,\cS']_{\can}$
 is separately continuous in each variable, and
 vanishes if and only if $\cS=\cS'\in\bP^1$ (see \cite[Proposition 4.10]{BR10}).
 \end{fact} 

 We normalize the Laplacian $\Delta$ on $\sP^1$ 
 so that for every $\cS\in\sP^1$,
 \begin{gather}
 \Delta\log[\cdot,\cS]_{\can}=\delta_{\cS}-\Omega_{\can}\label{eq:Hsia}
 \end{gather}
 on $\sP^1$ (for the construction of $\Delta$ on $\sP^1$ for non-archimedean $K$,
 see \cite[\S 5]{BR10}, \cite[\S7.7]{FJbook}, \cite[\S 3]{Thuillierthesis}:
 in \cite{BR10} the opposite sign convention on $\Delta$ is adopted).

Since we are interested in dynamics of rational functions, we introduce
only Berkovich (open or closed) connected affinoids in $\sP^1$.

\begin{fact}\label{th:affinoid}
 Suppose that $K$ is non-archimedean. A Berkovich closed disk $\sD$ is
 either $\{\cS\in\sP^1\setminus\{\infty\}:|\cS-w|\le r\}$ or 
 $\{\cS\in\sP^1\setminus\{\infty\}:|\cS-w|\ge r\}\cup\{\infty\}$
 for some $w\in\bP^1\setminus\{\infty\}$ and some $r\ge 0$, and
 is said to be {\itshape strict $($or rational$)$} if $r\in|K|$. Similarly,
 a Berkovich open disk is either $\{\cS\in\sP^1\setminus\{\infty\}:|\cS-w|<r\}$ or 
 $\{\cS\in\sP^1\setminus\{\infty\}:|\cS-w|>r\}\cup\{\infty\}$
 for some $w\in\bP^1\setminus\{\infty\}$ and some $r\ge 0$, and
 is said to be strict (or rational) if $r\in|K|$. 

 A Berkovich open (resp.\ closed) {\itshape connected} affinoid $U$ in $\sP^1$
 is the intersection of finitely many Berkovich open (resp.\ closed) disks
 and $\sP^1$, and is said to be strict if in addition all
 the Berkovich open (resp.\ closed) disks determining $U$ 
 are strict (or rational). 

 A Berkovich open connected affinoid is also called either a simple domain or
 an open fundamental domain in $\sP^1$.
 The set of all strict Berkovich open connected affinoids generates the topology
 of $\sP^1$ (cf.\ \cite[\S 2.6]{BR10}, \cite[\S 6]{Benedetto10}, \cite[\S 2.1]{FR09}).
For non-archimedean $K$,
the relative topology of $\bP^1$ in $\sP^1$ agrees with 
the metric topology on $\bP^1$ induced by the chordal distance on $\bP^1$.
Both $\bP^1$ and $\sP^1\setminus\bP^1$
are dense in $\sP^1$. 
\end{fact}

From rigid analysis,

\begin{definition}
 For non-archimedean $K$,
 a closed (resp.\ open) connected affinoid in $\bP^1$ is the intersection between
 a Berkovich closed (resp.\ open) connected affinoid $U$ in $\sP^1$ and $\bP^1$, and said to be
 {\itshape strict} if $U$ is strict. 
 A ($K$-valued) holomorphic function $T$ on a strict closed connected affinoid $V$ in $\bP^1$
 is defined by a uniform limit on $V$ (with respect to $[\cdot,\cdot]$)
 of a sequence of rational functions on $\bP^1$
 with no pole in $V$. By definition, 
 a holomorphic function $T$ on an open subset $D$ in $\bP^1$
 is a function on $D$ which restricts to a holomorphic function on any strict 
 closed connected affinoid $V$ in $D$.
\end{definition}

\begin{fact}\label{th:holo}
For non-archimedean $K$,
the modulus $|T|$ of a holomorphic function $T$
on a strict closed connected affinoid $V$ in $\bP^1$
attains both its maximum and minimum values on $V$ (the maximum modulus principle,
cf.\ \cite[\S 6.2.1, \S 7.3.4]{BGR}).
If in addition $T$ is non-constant, then $T$ has at most finitely many zeros in $V$
(this follows from the Weierstrass preparation theorem, cf.\ \cite[Theorem 3.5]{Benedetto10}).
\end{fact}

Let $\phi\in K(z)$ be a rational function on $\bP^1$.
For non-archimedean $K$, the analytic structure on $\sP^1$ induces
the extended action on $\sP^1$ of $\phi$. For non-constant $\phi$,
the extended action of $\phi$ on $\sP^1$ is continuous, open, surjective, and fiber-discrete,
and preserves $\bP^1$ and $\sP^1\setminus\bP^1$ 
(see \cite[Corollaries 9.9, 9.10]{BR10}, \cite[\S2.2]{FR09}).

\begin{fact}\label{th:rigid}
Suppose that $K$ is non-archimedean and that $\phi$ is non-constant.
Then $\phi$ maps a Berkovich disk (resp.\ Berkovich connected affinoid) onto
either $\sP^1$ or a Berkovich disk (resp.\ Berkovich connected affinoid),
preserving their openness, closedness, and strictness. 
Each component $U$ of $\phi^{-1}(V)$ for any Berkovich connected affinoid $V$
is a Berkovich connected affinoid, and the restriction $\phi:U\to V$ is proper and 
surjective
(\cite[Corollary 9.11, Lemma 9.12]{BR10},
\cite[Propositions 6.13]{Benedetto10}, \cite[Proposition 2.6]{Juan03}).
 The local (algebraic) degree $\deg_{z_0}\phi\in\bN$ of $\phi$ 
 at each $z_0\in\bP^1$ also uniquely extends
 to the function $\deg_{\cS}\phi\in\bN$
 for all $\cS\in\sP^1$ so that
 for any Berkovich open connected affinoid $V$ and every component $U$ of $\phi^{-1}(V)$,
 the function
\begin{gather*}
 V\ni\cS_0\mapsto\sum_{\cS\in\phi^{-1}(\cS_0)\cap U}\deg_{\cS}\phi\in\bN
\end{gather*}
 is constant (\cite[\S 2, \S 9]{BR10} and \cite[\S 2.1, Proposition-D\'efinition 2.1]{FR09}.
 See also \cite[\S 6.3]{Benedetto10}, \cite[\S 4]{Jonsson12}). 
 We denote this constant by $\deg(\phi:U\to V)$.
\end{fact}

If $\deg\phi>0$, then 
the extended $\phi:\sP^1\to\sP^1$ and
the local degree $\deg_{\cS}\phi$ of $\phi$ at each $\cS\in\sP^1$ induce 
a push-forward $\phi_*$ and
pullback $\phi^*$ on the space of continuous functions on $\sP^1$,
on the space of $\delta$-subharmonic functions on $\sP^1$
(functions on $\sP^1$ which can locally be written as the difference of two subharmonic functions), 
and on the space of Radon measures on $\sP^1$
(see \cite[\S9.4, \S9.5]{BR10}, \cite[\S2.2]{FR09}).
When $\deg\phi=0$, for a Radon measure $\mu$ on $\sP^1$, we set $\phi^*\mu:=0$
by convention. 
It is fundamental that for each non-constant $\phi$, 
the Laplacian $\Delta$ behaves functorially under $\phi^*$ in that for any $\delta$-subharmonic function
$h$ on $\sP^1$,
\begin{gather*}
 \Delta\phi^*h=\phi^*\Delta h
\end{gather*} 
on $\sP^1$ (for non-archimedean $K$, see
\cite[\S9.5]{BR10}, \cite[\S 2.4]{FR09}).

\begin{definition}\label{th:dynGreen}
 A lift $F_{\phi}=((F_{\phi})_0,(F_{\phi})_1):K^2\to K^2$ of $\phi$ is a homogeneous polynomial endomorphism
 of $K^2$ such that
 \begin{gather*}
 \pi\circ F_{\phi}=\phi\circ\pi
 \end{gather*} 
 and that $F_{\phi}^{-1}(0)=\{0\}$ if $\deg\phi>0$.
 Such an $F_{\phi}$ is unique up to scaling by an element of $K^*= K\setminus\{0\}$, 
 and $\deg F_{\phi}=\deg\phi$. The function 
 \begin{gather*}
  \log\|F_{\phi}\|-(\deg\phi)\log\|\cdot\|
 \end{gather*} 
 on $K^2\setminus\{0\}$ descends to one on $\bP^1$, which in turn
 extends continuously to a function $T_{F_{\phi}}:\sP^1\to\bR$ satisfying
 \begin{gather}
 \Delta T_{F_{\phi}}=\phi^*\Omega_{\can}-(\deg\phi)\Omega_{\can}\label{eq:cohomological}
 \end{gather}
 on $\sP^1$; indeed, for each $w\in\bP^1\setminus\{\infty\}$,
 since $|\cdot-w|=[\cdot,w]_{\can}[\cdot,\infty]_{\can}^{-1}[w,\infty]^{-1}$ on $\sP^1$,
 we have $\Delta\log|\cdot-w|=\delta_w-\delta_{\infty}$ on $\sP^1$.
 The homogeneous polynomial $(F_{\phi})_0(p_0,p_1)\in K[p_0,p_1]$
 factors into $\deg\phi$ homogeneous linear factors in $K[p_0,p_1]$. 
 Hence the function $\log|(F_{\phi})_0(p_0,p_1)|-(\deg\phi)\log|p_0|$ on $K^2\setminus\{0\}$
 descends to one on $\bP^1$, which in turn
 extends to a $\delta$-subharmonic function $S_{F_{\phi}}$ on $\sP^1$ satisfying
 $\Delta S_{F_{\phi}}=\phi^*\delta_{\infty}-(\deg\phi)\delta_{\infty}$
 on $\sP^1$. This yields \eqref{eq:cohomological}
 since
 $T_{F_{\phi}}=S_{F_{\phi}}-\log[\phi(\cdot),\infty]_{\can}+(\deg\phi)\log[\cdot,\infty]_{\can}$
 on $\sP^1$.
\end{definition}   

Let $\phi_i\in K(z)$, $i\in\{1,2\}$,
be rational functions on $\bP^1$ of degree $d_i$.
We call the following extension $[\phi_1,\phi_2]_{\can}$ to $\sP^1$
of the function $z\mapsto [\phi_1(z),\phi_2(z)]$ on $\bP^1$
{\itshape the proximity function of $\phi_1$ and $\phi_2$ on $\sP^1$}.

\begin{proposition}\label{th:extension}
For each $n\in\bN$, 
the function $[\phi_1(\cdot),\phi_2(\cdot)]$ on $\bP^1$ extends continuously to a function 
$[\phi_1,\phi_2]_{\can}(\cdot)$ on $\sP^1$ which takes its values in $[0,1]$ and,
if $\phi_1\not\equiv\phi_2$ and $\max\{d_1,d_2\}>0$, satisfies
 \begin{gather}
  \Delta\log[\phi_1,\phi_2]_{\can}(\cdot)=\sum_{w\in\bP^1:\phi_1(w)=\phi_2(w)}\delta_w
  -\phi_1^*\Omega_{\can}-\phi_2^*\Omega_{\can}.\label{eq:movingext}
 \end{gather}
 Here the sum $\sum_{w\in\bP^1:\phi_1(w)=\phi_2(w)}\delta_w$ takes into account
 the multiplicity of each root of $\phi_1=\phi_2$ in $\bP^1$.
\end{proposition}

\begin{proof}
 Let $F_1$ and $F_2$ be lifts of $\phi_1$ and $\phi_2$, respectively. Then
 there are points $q_j=q_j^{F_1,F_2}\in K^2\setminus\{0\}$ $(j=1,\ldots,d_1+d_2)$ such that
\begin{gather*}
 F_1(p)\wedge F_2(p)=\prod_{j=1}^{d_1+d_2}(p\wedge q_j)
\end{gather*}
 on $K^2$. Here, $\pi(q_j)$ is a root of $\phi_1=\phi_2$ in $\bP^1$
 for each $j\in\{1,\ldots,d_1+d_2\}$.
 On $\bP^1$,
\begin{gather}
 \log[\phi_1(\cdot),\phi_2(\cdot)]
=\sum_{j=1}^{d_1+d_2}(\log[\cdot,\pi(q_j)]+\log\|q_j\|)
-T_{F_1}|\bP^1-T_{F_2}|\bP^1,\label{eq:factor}
\end{gather}
 where $T_{F_i} = \log\|F_i\|-d_i\log\|\cdot\|$
(extended continuously to $\sP^1$), $i\in\{1,2\}$, is the function introduced in Definition 
\ref{th:dynGreen}.
The right hand side of \eqref{eq:factor} extends 
$[\phi_1(\cdot),\phi_2(\cdot)]$ on $\bP^1$ to 
$[\phi_1,\phi_2]_{\can}(\cdot)$ on $\sP^1$ continuously 
so that
\begin{gather*}
 \log[\phi_1,\phi_2]_{\can}(\cdot)
=\sum_{j=1}^{d_1+d_2}(\log[\cdot,\pi(q_j)]_{\can}+\log\|q_j\|)-T_{F_1}-T_{F_2}
\end{gather*}
 on $\sP^1$ (see Fact \ref{th:separately}),
 and satisfies \eqref{eq:movingext} from \eqref{eq:Hsia} and \eqref{eq:cohomological}.
 The density of $\bP^1$ in $\sP^1$ 
 implies that $[\phi_1,\phi_2]_{\can}(\cdot)\in[0,1]$ on $\sP^1$.
\end{proof}

\begin{remark}[a discontinuity of {$[\phi_1(\cdot),\phi_2(\cdot)]_{\can}$}]\label{th:incoincidence}
 If $\phi_2\equiv a\in\bP^1$ on $\bP^1$, then 
 $[\phi_1(\cdot),a]_{\can}$ coincides with $[\phi_1,a]_{\can}(\cdot)$ 
 since they are continuous on $\sP^1$ and
 identical on the dense subset $\bP^1$ in $\sP^1$. We point out that
 if $K$ is non-archimedean and both $\phi_1$ and $\phi_2$ are non-constant, 
 then $[\phi_1(\cdot),\phi_2(\cdot)]_{\can}$, which is the evaluation of $[\cS_1,\cS_2]_{\can}$
 at $\cS_1=\phi_1(\cdot)$ and $\cS_2=\phi_2(\cdot)$ in $\sP^1$, is not always continuous on $\sP^1$,
 so is not always identical with $[\phi_1,\phi_2]_{\can}(\cdot)$. This discrepancy
 seems to have been overlooked in the proof of Theorem \ref{th:moving} in \cite[\S 3.4]{FR09}.

 An example is $\phi_1=\phi_2=\Id_{\bP^1}$; see Fact \ref{th:separately}.
 More generally,
 let $\phi_1$ and $\phi_2$ be non-constant polynomials such that 
 $\phi_1(0)=\phi_2(0)=0$ and that $\phi_1'(0)=\phi_2'(0)\neq 0$.
 Fix $r>0$ small enough that on $\{z\in K:|z|<2r\}$,
 \begin{gather*}
  [\phi_1(z),\phi_2(z)]_{\can}=[\phi_1(z),\phi_2(z)]=|\phi_1(z)-\phi_2(z)|\le\frac{1}{2}|\phi_1'(0)|r,
 \end{gather*} 
 and that for the point $\cS_r\in\sP^1\setminus\bP^1$ determined by
 the disk $\{z\in K:|z|\le r\}$,
\begin{gather*}
 [\phi_1(\cS_r),\phi_2(\cS_r)]_{\can}=\diam(\phi_1(\cS_r)\wedge\phi_2(\cS_r))=
 |\phi_1(\cS_r)|=|\phi_1'(0)|r>0. 
\end{gather*}
 Since any open neighborhood of $\cS_r$ in $\sP^1$ intersects
 $\{z\in K:|z|<2r\}$, we have
 $\liminf_{\cS\to\cS_r}[\phi_1(\cS),\phi_2(\cS)]_{\can}\le |\phi_1'(0)|r/2<[\phi_1(\cS_r),\phi_2(\cS_r)]_{\can}$.
 Hence the function $[\phi_1(\cdot),\phi_2(\cdot)]_{\can}$ on $\sP^1$ is not continuous at $\cS_r$.
\end{remark}

Let $f\in K(z)$ be a rational function on $\bP^1$ of degree $d>1$, and let $F$ be a lift of $f$. 

\begin{definition}\label{th:cohomological}
 The dynamical Green function of $F$ on $\sP^1$ is
 \begin{gather}
  g_F:=\sum_{n=0}^{\infty}\frac{1}{d^n}(f^n)^*\left(\frac{1}{d}T_F\right)
  =\lim_{n\to\infty}\frac{1}{d^n}T_{F^n}\in\bR\label{eq:green},
 \end{gather}
 which converges uniformly on $\sP^1$ (\cite[\S10.1]{BR10}, \cite[\S3.1]{FR09}). 
\end{definition}

 The function $g_F$ is continuous on $\sP^1$.
 For every $k\in\bN$, we have $g_{F^k}=g_F$.
 For an arbitrary lift of $f$, given by $cF$ for some $c\in K^*$,
 we have $g_{cF}=g_F+(\log|c|)/(d-1)$. 

\begin{definition}
 The probability Radon measure
 \begin{gather}
 \mu_f:=\Delta g_F+\Omega_{\can}=\lim_{n\to\infty}\frac{1}{d^n}(f^n)^*\Omega_{\can}\label{eq:equilibrium}
 \end{gather}
 on $\sP^1$ is called the equilibrium measure of $f$ on $\sP^1$. 
 Here the last limit is a weak one on $\sP^1$.
\end{definition}

\begin{fact}
 By the continuity of $g_F$, the measure $\mu_f$ has no atoms in $\bP^1$. 
 Moreover, $\mu_f$ is both balanced and invariant under $f$ in that
 \begin{gather}
  f^*\mu_f=(\deg f)\mu_f\quad\text{and}\quad f_*\mu_f=\mu_f,\label{eq:balanced}
 \end{gather}
 respectively (see \cite[\S10]{BR10}, \cite[\S2]{ChambertLoir06}, \cite[\S3.1]{FR09}
 for non-archimedean $K$). 
\end{fact}

We define the $F$-kernel on $\sP^1$ to be
\begin{gather*}
 \Phi_F(\cS,\cS'):=\log[\cS,\cS']_{\can}-g_F(\cS)-g_F(\cS')
\end{gather*}
for $\cS,\cS'\in\sP^1$.
The function $\Phi_F$ is upper semicontinuous on $\sP^1\times\sP^1$, and
for each $\cS\in\sP^1\setminus\bP^1$, $\Phi_F(\cS,\cdot)$ is continuous on $\sP^1$
(see Fact \ref{th:separately}).
The comparison
\begin{gather*}
 \sup_{(\cS,\cS')\in\sP^1\times\sP^1}|\Phi_F(\cS,\cS')
-\log[\cS,\cS']_{\can}|\le2\sup_{\sP^1}g_F<\infty
 \end{gather*} 
holds, and from \eqref{eq:Hsia} and \eqref{eq:equilibrium},
$\Delta\Phi_F(\cdot,\cS)=\delta_{\cS}-\mu_f$ for each $\cS\in\sP^1$.
For a Radon measure $\mu$ on $\sP^1$,
the $F$-potential on $\sP^1$ and the $F$-energy of $\mu$ are 
\begin{gather*}
 U_{F,\mu}(\cdot):=\int_{\sP^1}\Phi_F(\cdot,\cS')\rd\mu(\cS'),\\
 I_F(\mu):=\int_{\sP^1} U_{F,\mu}\rd\mu,
\end{gather*}
respectively (see also \cite[\S 8.10]{BR10}, \cite[\S 2.4]{FR09}).
The function $U_{F,\mu}$ is upper semicontinuous on $\sP^1$ and 
satisfies the following
continuity property: for every $z_0\in\bP^1\setminus\{\infty\}$ and
every $r\ge 0$, if $\cS_r(z_0)$
is the point in $\sP^1$ corresponding to the disk $\cB_r(z_0):=\{z\in K:|z-z_0|\le r\}$,
we have
\begin{gather}
 \lim_{r\to 0}U_{F,\mu}(\cS_r(z_0))=U_{F,\mu}(z_0)\label{eq:continuity}
\end{gather}
(see \cite[Proposition 6.12]{BR10}).
By Fubini's theorem,
\begin{gather*}
 \Delta U_{F,\mu}=\mu-\mu(\sP^1)\mu_f.
\end{gather*}

A probability Radon measure $\mu$ on $\sP^1$ is called an 
{\itshape $F$-equilibrium mass distribution on $\sP^1$} if
the $F$-energy $I_F(\mu)$ of this $\mu$ equals
\begin{gather*}
 V_F:=\sup\{I_F(\nu):\nu\text{ is a probability Radon measure on }\sP^1\},
\end{gather*}
which is $>-\infty$ since $I_F(\Omega_{\can})>-\infty$.

We recall Baker and Rumely's characterization of $\mu_f$ 
as the unique solution of a Gauss variational problem;
see \cite[Theorem 8.67 and Proposition 8.70]{BR10} for non-archimedean
$K$. For a discussion of the Gauss variational problem, see e.g.\ \cite{ST97}.

\begin{lemma}\label{th:Frostman}
 There is a unique $F$-equilibrium mass distribution on $\sP^1$, which coincides with
 the equilibrium measure $\mu_f$ of $f$. Indeed, on $\sP^1$, 
 \begin{gather}
  U_{F,\mu_f}\equiv V_F.\label{eq:identity}
 \end{gather} 
\end{lemma}

The functions
$\Phi_F,U_{F,\mu}$ and $g_F$ depend on the lift $F$ of $f$.
We will now introduce more canonical functions $\Phi_f$, $U_{\mu}$, and $g_f$, 
which do not depend on the choice of the lift $F$.
The $f$-kernel on $\sP^1$ 
(the negative of the Arakelov Green function for $f$ in \cite[\S10.2]{BR10}) is
\begin{gather*}
 \Phi_f:=\Phi_F-V_F.
\end{gather*}
It is independent of the choice of $F$. 
For each Radon measure $\mu$ on $\sP^1$, 
we define the $f$-potential
\begin{gather*}
 U_{\mu}:=\int_{\sP^1}\Phi_f(\cdot,\cS')\rd\mu(\cS')
\end{gather*}
on $\sP^1$. We still have $\Delta U_{\mu}=\mu-\mu(\sP^1)\mu_f$. 
From Lemma \ref{th:Frostman},

\begin{lemma}\label{th:canonical}
For each Radon measure $\mu$ on $\sP^1$, we have
$U_{\mu}\ge 0$ on $\supp\mu$ if and only if $\mu=\mu_f$.
Moreover, $U_{\mu_f}\equiv 0$ on $\sP^1$.
\end{lemma}

The dynamical Green function $g_f$ of $f$ (a canonical version of $g_F$) is
defined as
\begin{gather*}
 g_f(\cS):=g_F(\cS)+\frac{1}{2}V_F=\frac{1}{2}(\log[\cS,\cS]_{\can}-\Phi_f(\cS,\cS)),
\end{gather*}
which is independent of the choice of $F$ and
still satisfies 
\begin{gather}
 \Delta g_f=\mu_f-\Omega_{\can}.\label{quasipotential} 
\end{gather}
For every $(\cS,\cS')\in\sP^1\times\sP^1$,
\begin{gather*}
 \Phi_f(\cS,\cS')=\log[\cS,\cS']_{\can}-g_f(\cS)-g_f(\cS').
\end{gather*}

Our definition \eqref{eq:equilibrium} of $\mu_f$ 
agrees with Favre and Rivera-Letelier's \cite[Proposition-D\'efinition 3.2]{FR09}:

\begin{lemma}\label{th:recover}
For every $\cS\in\sP^1\setminus\bP^1$, weakly on $\sP^1$,
\begin{gather*}
  \lim_{k\to\infty}\frac{(f^n)^*\delta_{\cS}}{d^n}=\mu_f.
\end{gather*}
\end{lemma} 

\begin{proof}
 For every $\cS\in\sP^1$ and every $n\in\bN$, from the balanced property $f^*\mu_f=d\cdot\mu_f$,
 \begin{gather*}
  \Delta\Phi_f(f^n(\cdot),\cS)
  =(f^n)^*(\delta_{\cS}-\mu_f)=(f^n)^*\delta_{\cS}-d^n\mu_f
 \end{gather*}
 on $\sP^1$. Suppose that $\cS\in\sP^1\setminus\bP^1$. Then since $[\cS,\cS]_{\can}>0$,
 \begin{gather*}
  \sup_{\cS'\in\sP^1}|\Phi_f(f^n(\cS'),\cS)|\le|\log[\cS,\cS]_{\can}|+2\sup_{\sP^1}g_f<\infty,
 \end{gather*}
 so $\lim_{n\to\infty}\Phi_f(f^n(\cdot),\cS)/d^n=0$ uniformly on $\sP^1$.
 By the continuity of $\Delta$
 on uniformly convergent sequences of $\delta$-subharmonic functions 
 (for non-archimedean $K$, see \cite[Corollary 5.39]{BR10},
 \cite[Proposition 2.17]{FR09}), as $n\to\infty$,
 \begin{gather*}
  \frac{(f^n)^*\delta_{\cS}}{d^n}-\mu_f=\Delta\frac{1}{d^n}\Phi_f(f^n(\cdot),\cS)\to 0
 \end{gather*}
 weakly on $\sP^1$. 
\end{proof}

The Berkovich Fatou set $\sF(f)$ of $f$ is by definition
$\sP^1\setminus\sJ(f)$, which is open in $\sP^1$.
A Berkovich Fatou component $W$ of $f$ is a component of $\sF(f)$.
Given such a $W$, $f(W)$ is also a Berkovich Fatou component of $f$, and
so is each component of $f^{-1}(W)$.
We call $W$ a cyclic Berkovich Fatou component of $f$ if for some $p\in\bN$,
$f^p(W)=W$. 

For archimedean $K$, the classification of cyclic Fatou components
(immediate (super)attractive basins of attracting cycles,
immediate attractive basins of parabolic cycles, Siegel disks,
and Herman rings) of $f$
is essentially due to Fatou (cf.\ \cite[Theorem 5.2]{Milnor3rd}).
The following is its non-archimedean counterpart
due to Rivera-Letelier; see \cite[Proposition 2.16]{FR09} and 
its {\itshape esquisse de d\'emonstration}, and also \cite[Remark 7.10]{Benedetto10}.

\begin{theorem}\label{th:classification}
Suppose that $K$ is non-archimedean. Then each cyclic Berkovich Fatou component 
$W$ of $f$ satisfies either that $W$ contains an attracting periodic point of $f$ in $W\cap\bP^1$
$($attracting case$)$ or that $\deg(f^p:W\to W)=1$ for some $p\in\bN$ satisfying $f^p(W)=W$.
Moreover, only one case occurs.
In the former case, $W$ is called an immediate 
$($super$)$attractive
basin of $f$, and in the latter case, $W$ is called a singular domain of $f$.
\end{theorem}

All of $E(f), \sJ(f)$, $\sF(f)$, and $\supp\mu_f$ are completely invariant under $f$. 
Here, a subset $E$ in $\sP^1$ is said to be
completely invariant under $f$ if $f(E)\subset E$ and $f^{-1}(E)\subset E$.
The following equality is fundamental.

\begin{theorem}\label{th:inclusion}
$\sJ(f)=\supp\mu_f$. 
Moreover, for each $a\in E(f)$, 
no weak limit point of $(\nu_n^a)$ on $\sP^1$ agrees with $\mu_f$.
\end{theorem}

\begin{proof}
Since $\mu_f$ has no atoms in $\bP^1$ and $E(f)$ is a countable subset in $\bP^1$,
$\supp\mu_f\not\subset E(f)$. Then
$\sJ(f)\subset\overline{\bigcup_{n\in\bN}f^{-n}((\supp\mu_f)\setminus E(f))}$,
which is contained in $\supp\mu_f$. Hence $\sJ(f)\subset\supp\mu_f$.

For archimedean $K$, 
$\Omega_{\can}$ is the normalized Fubini-study metric on $\sP^1=\bP^1$.
By Marty's theorem \cite[Th\'eor\`eme 5]{Marty31}, which is an infinitesimal version of
Montel's theorem, $\sF(f)$ agrees with the maximal open subset in $\bP^1$
where the family of chordal derivatives of $f^n$, $n\in\bN$,
\begin{gather*}
\left\{\bP^1\ni z\mapsto \sqrt{\frac{(f^n)^*\Omega_{\can}}{\Omega_{\can}}(z)}=\lim_{w\to z}\frac{[f^n(z),f^n(w)]}{[z,w]}\in[0,\infty):n\in\bN\right\} 
\end{gather*}
is locally uniformly bounded.
Hence by the definition \eqref{eq:equilibrium} of $\mu_f$, 
we have $\sF(f)\subset\sP^1\setminus\supp\mu_f$,
i.e., $\supp\mu_f\subset\sJ(f)$.

Suppose that $K$ is non-archimedean.
If $\sJ(f)\subset\bP^1$, then $\sF(f)$ is itself the unique Berkovich 
Fatou component of $f$, which is completely invariant under $f$.
Since $\deg(f:F(f)\to F(f))=\deg f>1$, by Theorem \ref{th:classification},
$\sF(f)$ is the immediate attractive basin of an attracting fixed point $a\in\bP^1$. 
Since $\cS_{\can}\in\sP^1\setminus\bP^1\subset\sF(f)\setminus\{a\}$, 
we have $\bigcap_{N\in\bN}\overline{\bigcup_{n\ge N}f^{-n}(\cS_{\can})}\subset\partial\sF(f)=\sJ(f)$.
Moreover, since $\Omega_{\can}=\delta_{\cS_{\can}}$ in this case,
by the definition \eqref{eq:equilibrium} of $\mu_f$,
$\supp\mu_f\subset\bigcap_{N\in\bN}\overline{\bigcup_{n\ge N}f^{-n}(\cS_{\can})}$.
Hence $\supp\mu_f\subset\sJ(f)$.
Finally, if $\sJ(f)\not\subset\bP^1$, then by Lemma \ref{th:recover},
we have $\supp\mu_f\subset\overline{\bigcup_{n\in\bN}f^{-n}(\sJ(f)\setminus\bP^1)}$, 
which is contained in $\sJ(f)$. 

Hence we have $\supp\mu_f\subset\sJ(f)$ in both archimedean and non-archimedean cases, 
and the proof of the former assertion is complete.

Recall that for any $a\in E(f)$, the backward orbit $\bigcup_{n\in\bN}f^{-n}(a)$ is finite and
contained in $\sF(f)$. Hence any weak limit point $\nu=\lim_{j\to\infty}\nu_{n_j}^a$
has its support in $\sF(f)$, so $\nu\neq\mu_f$ by the former assertion.
\end{proof}

Finally, for a rational function $f\in K(z)$ on $\bP^1$ of degree $d>1$ and
a rational function $a\in K(z)$ on $\bP^1$, 
we introduce the (logarithmic) proximity function $\log[f^n,a]_{\can}(\cdot)$
of $f^n(\cdot)$ and $a(\cdot)$ weighted by $g_f$:
\begin{gather*}
 \Phi(f^n,a)_f(\cdot):=\log[f^n,a]_{\can}(\cdot)-g_f\circ f^n-g_f\circ a.
\end{gather*}
The function $\Phi(f^n,a)_f(\cdot)$ extends the function 
$z\mapsto \Phi_f(f^n(z),a(z))$ on $\bP^1$ continuously to $\sP^1$
and plays a crucial role in the rest of the paper. 
It agrees with $\Phi_f(f^n(\cdot),a)$ when $a$ is constant.
For each $n\in\bN$, we have the comparison
\begin{gather}
 \sup_{\sP^1}|\Phi(f^n,a)_f(\cdot)-\log[f^n,a]_{\can}(\cdot)|\le
 \sup_{\sP^1}|2g_f|<\infty.\label{eq:movingcomparison}
\end{gather} 

\begin{lemma}[{cf. \cite[(1.4)]{Sodin92}}]\label{th:Riesz}
 For every $n\in\bN$, 
 \begin{gather}
  \frac{1}{d^n}\Phi(f^n,a)_f(\cdot)
  =U_{(1+(\deg a)/d^n)\nu_n^a}-\frac{1}{d^n}U_{a^*\mu_f}
+\frac{1}{d^n}\int_{\sP^1}\Phi(f^n,a)_f(\cdot)\rd\mu_f\label{eq:original}
 \end{gather}
 on $\sP^1$. Similarly, the function
 $U_{a^*\mu_f}=a^*g_f+U_{a^*\Omega_{\can}}-\int_{\sP^1}(a^*g_f)\rd\mu_f$
 is continuous $($hence bounded$)$ on $\sP^1$.
\end{lemma}

\begin{proof}
 For each $n\in\bN$, from \eqref{eq:movingext} and \eqref{quasipotential},
\begin{gather*}
 \Delta\Phi(f^n,a)_f(\cdot)
 =(d^n+\deg a)\nu_n^a-(f^n)^*\mu_f-a^*\mu_f,
\end{gather*}
and using the balanced property $f^*\mu_f=d\cdot\mu_f$, we have
\begin{gather*}
  \Delta\Phi(f^n,a)_f(\cdot)
 =\Delta(U_{(d^n+\deg a)\nu_n^a}-U_{a^*\mu_f}).
\end{gather*}
 Hence the function
 \begin{gather*}
  \frac{1}{d^n}\Phi(f^n,a)_f(\cdot)
  -(U_{(1+(\deg a)/d^n)\nu_n^a}(\cdot)-\frac{1}{d^n}U_{a^*\mu_f}(\cdot))
 \end{gather*}
 is constant on $\sP^1$ (for non-archimedean $K$, this holds on $\sP^1\setminus\bP^1$
 by a basic property of $\Delta$ (see \cite[Lemma 5.24]{BR10}, \cite[\S 2.4]{FR09}) and indeed
 on $\sP^1$ by continuity \eqref{eq:continuity}). 
We determine the constant by
integrating this against $\rd\mu_f$ on $\sP^1$: by Fubini's
theorem and the fact that $U_{\mu_f}\equiv 0$, the integrals of the second and third terms
in $\rd\mu_f$ vanish. Hence \eqref{eq:original} holds.

Similarly, from $\Delta U_{a^*\mu_f}=a^*\mu_f-(\deg a)\mu_f=\Delta(a^*g_f+U_{a^*\Omega_{\can}})$,
the function $U_{a^*\mu_f}-(a^*g_f+U_{a^*\Omega_{\can}})$ is constant on
$\sP^1$. The constant is determined by integrating this function against $\rd\mu_f$ on $\sP^1$.
\end{proof}

\section{Proof of Theorem \ref{th:julia}}
\label{sec:equidist}

Let $K$ be an algebraically closed field of any characteristic and complete with respect
to a non-trivial absolute value.
Let $f\in K(z)$ be a rational function on $\bP^1=\bP^1(K)$ of degree $d>1$, and
$a\in K(z)$ a rational function on $\bP^1$.
Let $(n_j)$ be a sequence in $\bN$
tending to $\infty$, and $\nu$ be any weak limit of a subsequence of
$(\nu_{n_j}^a)$ on $\sP^1=\sP^1(K)$. 
This is a probability Radon measure on $\sP^1$, and
the equidistribution property \eqref{eq:subgeneral} is equivalent to 
\begin{gather}
 \nu=\mu_f.\tag{\ref{eq:subgeneral}'}
\end{gather}
Taking a subsequence of $(n_j)$ if necessary, we can assume that
$\nu=\lim_{j\to\infty}\nu_{n_j}^a$ weakly on $\sP^1$ and that the limit
\begin{gather}
 \lim_{j\to\infty}\frac{1}{d^{n_j}}\int_{\sP^1}\Phi(f^{n_j},a)_f\rd\mu_f\label{eq:limitmean}
\end{gather}
exists in $[-\infty,0]$. 

\begin{lemma}\label{th:limsup}
 On $\sP^1$, 
 \begin{multline}
 \limsup_{j\to\infty}\frac{1}{d^{n_j}}\log[f^{n_j},a]_{\can}(\cdot)
=\limsup_{j\to\infty}\frac{1}{d^{n_j}}\Phi(f^{n_j},a)_f(\cdot)\\
\le U_{\nu}+\lim_{j\to\infty}\frac{1}{d^{n_j}}\int_{\sP^1}\Phi(f^{n_j},a)_f\rd\mu_f
\le\min\{U_{\nu},0\}.\label{eq:limsup}
\end{multline}
 Moreover, on $\sP^1\setminus\bP^1$,
\begin{gather}
 \lim_{j\to\infty}\frac{1}{d^{n_j}}\log[f^{n_j},a]_{\can}(\cdot)=U_{\nu}+\lim_{j\to\infty}\frac{1}{d^{n_j}}\int_{\sP^1}\Phi(f^{n_j},a)_f\rd\mu_f.\label{eq:lim}
\end{gather}
\end{lemma}

\begin{proof}
 By a cut-off argument, on $\sP^1$, 
\begin{gather}
 \limsup_{j\to\infty}U_{\nu_{n_j}^a}\le U_{\nu};\label{eq:cutoff}
\end{gather} 
indeed, for every $N\in\bN$,
 $U_{\nu_{n_j}^a}\le\int_{\sP^1}\max\{-N,\Phi_f(\cdot,\cS')\}\rd\nu_{n_j}^a(\cS')$ on $\sP^1$,
 and since for every $\cS\in\sP^1$,
 the function $\cS'\mapsto\max\{-N,\Phi_f(\cS,\cS')\}$ is continuous on $\sP^1$, we have
\begin{gather*}
 \limsup_{j\to\infty}U_{\nu_{n_j}^a}\le\int_{\sP^1}\max\{-N,\Phi_f(\cdot,\cS')\}\rd\nu(\cS')
\end{gather*} 
 on $\sP^1$. Taking $N\to\infty$, we obtain \eqref{eq:cutoff} by the monotone convergence theorem.

 On the other hand, for every $\cS\in\sP^1\setminus\bP^1$, 
 the function $\cS'\mapsto\Phi_f(\cS,\cS')$ is continuous on $\sP^1$,
 so we have $\lim_{j\to\infty}U_{\nu_{n_j}^a}=U_{\nu}$ on $\sP^1\setminus\bP^1$.

 By the comparison \eqref{eq:movingcomparison} and $[f^n,a]_{\can}\le 1$,
\begin{gather*}
\limsup_{j\to\infty}\frac{1}{d^{n_j}}\Phi(f^{n_j},a)_f(\cdot)
=\limsup_{j\to\infty}\frac{1}{d^{n_j}}\log[f^{n_j},a]_{\can}(\cdot)\le 0
\end{gather*}
 on $\sP^1$. 
 Now taking $\limsup_{j\to\infty}$ of (\eqref{eq:original} for $n=n_j$), 
 we have \eqref{eq:limsup} on $\sP^1$, and also \eqref{eq:lim} on $\sP^1\setminus\bP^1$.
\end{proof}

If $a$ is constant, then by convention,
we identify $a$ with its value in $\bP^1$.

\begin{lemma}\label{th:const}
 If $a$ is constant, then $\int_{\sP^1}\Phi_f(f^n(\cdot),a)\rd\mu_f=0$
 for every $n\in\bN$, and $U_{\nu}\ge 0$ on $\sJ(f)$.
\end{lemma}

\begin{proof}
Let $a\in\bP^1$. Then for every $n\in\bN$, by the invariance 
$f_*\mu_f=\mu_f$ and the fact that
$U_{\mu_f}\equiv 0$ on $\sP^1$, we have
\begin{gather*}
 \int_{\sP^1}\Phi_f(f^n(\cdot),a)\rd\mu_f=U_{(f^n)_*\mu_f}(a)=U_{\mu_f}(a)=0.
\end{gather*}
Hence by Fatou's lemma and \eqref{eq:limsup}, this implies that
\begin{multline*}
 0=\lim_{j\to\infty}\frac{1}{d^{n_j}}\int_{\sP^1}\Phi_f(f^{n_j}(\cdot),a)\rd\mu_f\\
\le\int_{\sP^1}\limsup_{j\to\infty}\frac{1}{d^{n_j}}\Phi_f(f^{n_j}(\cdot),a)\rd\mu_f\le
\int_{\{U_{\nu}<0\}\cap\sJ(f)}U_{\nu}\rd\mu_f.
\end{multline*}
Since $\sJ(f)\subset\supp\mu_f$ (by Theorem \ref{th:inclusion}),
$\{U_{\nu}<0\}\cap\sJ(f)=\emptyset$.
\end{proof}

We show the following counterpart of Lemma \ref{th:const} for non-constant $a$.

\begin{lemma}\label{th:negativeFatou}
 If $a$ is non-constant, then $U_{\nu}\ge 0$ on $\sJ(f)$.
\end{lemma}

\begin{proof}
Assume that $\{U_{\nu}<0\}\cap\sJ(f)\neq\emptyset$. Then 
since $\{U_f < 0\}$ is open,
\begin{gather*}
 \bigcup_{n\in\bN}f^n(\{U_{\nu}<0\}\cap\bP^1)
 =\left(\bigcup_{n\in\bN}f^n(\{U_{\nu}<0\})\right)\cap\bP^1\supset\bP^1\setminus E(f).
\end{gather*}
If there exists $z_1\in E(f)$, 
then $\bigcup_{n\in\bN}f^n(\{U_{\nu}<0\}\cap\bP^1)$ intersects the
immediate attractive basin of $z_1$, so by \eqref{eq:limsup}, $a\equiv z_1$. 
This contradicts that $a$ is non-constant, and we have $E(f)=\emptyset$.

Let $z_0$ be a fixed point of $f$ in $\bP^1=\bP^1\setminus E(f)$.
Then by the assumption $\{U_{\nu}<0\}\cap\sJ(f)\neq\emptyset$ and
the definition of $\sJ(f)$, 
$\sJ(f)\cap\{U_{\nu}<0\}\subset
(\bigcap_{\ell\in\bN}\overline{\bigcup_{j\ge\ell}f^{-n}(z_0)})\cap\{U_{\nu}<0\}$.
Hence if $\#(\bigcup_{n\in\bN}f^{-n}(z_0)\cap\{U_{\nu}<0\})<\infty$, then
$\sJ(f)\cap\{U_{\nu}<0\}$ is a non-empty and finite subset in $\bP^1$.
Since $\sJ(f)\subset\supp\mu_f$ (by Theorem \ref{th:inclusion}),
this contradicts that $\mu_f$ has no atoms in $\bP^1$.

Hence there is an $N\in\bN$ such that 
$f^{-N}(z_0)\cap\{U_{\nu}<0\}\not\subset a^{-1}(z_0)$ since 
$\#(\bigcup_{n\in\bN}f^{-n}(z_0)\cap\{U_{\nu}<0\})=\infty$ and $\# a^{-1}(z_0)<\infty$.
Let $z_{-N}\in(f^{-N}(z_0)\cap\{U_{\nu}<0\})\setminus a^{-1}(z_0)$. 
Then 
\begin{gather*}
 \limsup_{j\to\infty}\frac{1}{d^{n_j}}\log[f^{n_j}(z_{-N}),a(z_{-N})]
=\limsup_{j\to\infty}\frac{1}{d^{n_j}}\log[z_0,a(z_{-N})]=0,
\end{gather*}
which contradicts \eqref{eq:limsup} at $z_{-N}$ since $U_{\nu}(z_{-N})<0$.

Hence $\{U_{\nu}<0\}\cap\sJ(f)=\emptyset$, and the proof is complete.
\end{proof}

\begin{lemma}\label{th:movingvanish}
If \eqref{eq:weaklim} holds, then 
\begin{gather}
 \lim_{j\to\infty}\frac{1}{d^{n_j}}\int_{\sP^1}\Phi(f^{n_j},a)_f(\cdot)\rd\mu_f=0.\label{eq:varilonmoving}
\end{gather}
Indeed, \eqref{eq:varilonmoving} holds for every $a\in\bP^1$ without assuming \eqref{eq:weaklim}.
\end{lemma}

 \begin{proof}
If $a$ is constant, then this follows from the former assertion in Lemma \ref{th:const}
without assuming \eqref{eq:weaklim}.

Suppose that $a$ is non-constant. If \eqref{eq:weaklim} holds but 
\eqref{eq:varilonmoving} does not hold,
then by \eqref{eq:limsup} and $U_{\nu}=U_{\mu_f}\equiv 0$,
\begin{gather}
 \limsup_{j\to\infty}\frac{1}{d^{n_j}}\log[f^{n_j},a]_{\can}(\cdot)
\le U_{\nu}+\lim_{j\to\infty}\frac{1}{d^{n_j}}\int_{\sP^1}\Phi(f^{n_j},a)_f\rd\mu_f<0\label{eq:unifconv}
\end{gather}
on $\sP^1$. If there exists $z_1\in E(f)$, then \eqref{eq:unifconv} holds on the
immediate attractive basin of $z_1$, so $a\equiv z_1$.
This is a contradiction, and we have $E(f)=\emptyset$.

Let $z_0\in\bP^1=\bP^1\setminus E(f)$ be a fixed point of $f$.
Then $\infty>\#a^{-1}(z_0)<\#\bigcup_{n\in\bN}f^{-n}(z_0)=\infty$, so
there is an $N\in\bN$ such that 
$f^{-N}(z_0)\not\subset a^{-1}(z_0)$.
Let $z_{-N}\in f^{-N}(z_0)\setminus a^{-1}(z_0)$. 
Then 
\begin{gather*}
 \limsup_{j\to\infty}\frac{1}{d^{n_j}}\log[f^{n_j}(z_{-N}),a(z_{-N})]
=\limsup_{j\to\infty}\frac{1}{d^{n_j}}\log[z_0,a(z_{-N})]=0,
\end{gather*}
which contradicts \eqref{eq:unifconv} at $z_{-N}$.
\end{proof} 

We can now complete the proof of Theorem \ref{th:julia}.
 If (\ref{eq:support}) holds, then 
 by the latter assertion in Lemma \ref{th:const}, Lemma \ref{th:negativeFatou}, and
 Lemma \ref{th:canonical},
 the condition \eqref{eq:weaklim} holds. The reverse implication
 is by Theorem \ref{th:inclusion}.

 Suppose now that $K$ is non-archimedean. If (\ref{eq:weaklim}) holds, then
 $U_{\nu}=U_{\mu_f}\equiv 0$ on $\sP^1$, and  by \eqref{eq:lim} and
 Lemma \ref{th:movingvanish}, we have
\begin{gather*}
\lim_{j\to\infty}\frac{1}{d^{n_j}}\log[f^{n_j},a]_{\can}(\cdot)=U_{\nu}+\lim_{j\to\infty}\frac{1}{d^{n_j}}\int_{\sP^1}\Phi(f^{n_j},a)_f\rd\mu_f=0,
\end{gather*}
 i.e., (\ref{eq:hyperbolic}), on $\sP^1\setminus\bP^1$.
 Conversely, if (\ref{eq:hyperbolic}) holds on $\sP^1\setminus\bP^1$, then
 by \eqref{eq:limsup}, $\{U_{\nu}<0\}\setminus\bP^1=\emptyset$, so 
 $\{U_{\nu}<0\}=\emptyset$.
 Hence by Lemma \ref{th:canonical}, \eqref{eq:weaklim} holds. 
 If one $($so ultimately all$)$ of 
 \eqref{eq:subgeneral}, \eqref{eq:support} and \eqref{eq:hyperbolic} holds, 
 then by Lemma \ref{th:movingvanish} and
 the comparison \eqref{eq:movingcomparison},
 the final \eqref{eq:Valiron} holds; indeed,
 \eqref{eq:Valiron} holds for every $a\in\bP^1$ without assuming 
 \eqref{eq:subgeneral}, \eqref{eq:support} or \eqref{eq:hyperbolic}.

This completes the proof of Theorem \ref{th:julia}.

\section{Proof of Theorems \ref{th:FR}, \ref{th:general} and \ref{th:moving}}
\label{sec:general} 
 
We give some addenda to our argument in Section \ref{sec:equidist}.
Let $K$ be an algebraically closed field of arbitrary characteristic,
and complete with respect
to a non-trivial absolute value.
Let $f\in K(z)$ be a rational function on $\bP^1$ of degree $d>1$, 
and $a\in K(z)$ a rational function on $\bP^1$.
If $a$ is constant, we identify $a$ with its value in $\bP^1$.
Let $\nu=\lim_{j\to\infty}\nu_{n_j}^a$ be the weak limit of a subsequence $(\nu_{n_j}^a)$
of $(\nu_n^a)$ on $\sP^1=\sP^1(K)$.
Taking a subsequence of $(n_j)$ if necessary, we can assume that
the limit \eqref{eq:limitmean} exists in $[-\infty,0]$.

We first give a purely local proof of Theorem \ref{th:FR}
based on \eqref{eq:diophantine} and Lemma \ref{th:canonical}.

\begin{proof}[Proof of Theorem $\ref{th:FR}$]
 Under the assumption in Theorem \ref{th:FR}, we set $K=\bC_v$.
 The set of all points in $\bP^1(\overline{k})$ which are wandering
 under $f$ and, if in addition $a$ is
 non-constant, do not belong to $a^{-1}(E(f))$,
 is dense in $\sP^1$. Since $U_{\nu}$ is upper semicontinuous,
 the inequality \eqref{eq:limsup}, combined with the dynamical Diophantine approximation result
 (\ref{eq:diophantine}),
 implies that $U_{\nu}\ge 0$ on $\sP^1$. 
 Hence by Lemma \ref{th:canonical}, \eqref{eq:weaklim} holds.
\end{proof}

Next, we show Theorem $\ref{th:general}$. 

\begin{proof}[Proof of Theorem $\ref{th:general}$]
We will show that $(\supp\nu)\cap\{U_{\nu}< 0\}=\emptyset$. 
This means that, by Lemma \ref{th:canonical}, \eqref{eq:weaklim} will hold.

Suppose first that $a\in\sJ(f)\cap\bP^1$. 
Then by $f^{-1}(\sJ(f))=\sJ(f)$,
we have $\supp\nu\subset\sJ(f)$. Hence by Lemma \ref{th:const},
$(\supp\nu)\cap\{U_{\nu}< 0\}=\emptyset$.

Suppose that $a\in (\sF(f)\cap\bP^1)\setminus E(f)$.
By the upper semicontinuity of $U_{\nu}$, $\{U_{\nu}<0\}$ is open. 
From \eqref{eq:limsup},
\begin{gather}
\limsup_{j\to\infty}\frac{1}{d^{n_j}}\log[f^{n_j}(\cdot),a]_{\can}
\le U_{\nu}(\cdot)<0\label{eq:quantitative}
\end{gather}
on $\{U_{\nu}<0\}$. This implies
that $\lim_{j\to\infty}f^{n_j}=a$ on $\{U_{\nu}<0\}\cap\bP^1$,
so $\{U_{\nu}<0\}\cap\bP^1\subset\sF(f)$, and that 
the Berkovich Fatou component $W$ of $f$ containing $a$ is cyclic under $f$,
i.e., $f^p(W)=W$ for some $p\in\bN$.
Then from the classification of cyclic (Berkovich) Fatou components
(see Theorem \ref{th:classification} for non-archimedean $K$), it follows 
either that $a$ is the unique attracting fixed point of $f^p$ in $W$ (attracting case)
or $\deg(f^p:W\to W)=1$ (singular case). 
In the attracting case, 
by \eqref{eq:quantitative}, $a$ is the superattracting fixed point of $f^p$ in $W$
satisfying $\deg_a f^p=d^p$. This contradicts the assumption $a\in\bP^1\setminus E(f)$.

Hence the singular case occurs. 
Let $U$ be a component of $\{U_{\nu}<0\}$ and
put $N:=\min\{n\in\bN\cup\{0\}:f^n(U)\subset W\}$.
Then for every $n>N$, there is at most one root of
$f^{n-N}(\cdot)=a$ in $W$,
which is simple if exists. Hence
\begin{gather*}
 0\le\nu(U)\le\limsup_{j\to\infty}\frac{1\cdot d^N}{d^{n_j}}=0.
\end{gather*}
This implies that $(\supp\nu)\cap\{U_{\nu}<0\}=\emptyset$.
\end{proof}

\begin{remark}
 For a purely potential theoretical proof of Theorem \ref{th:general}
 for non-archimedean $K$, see \cite[\S 5]{Jonsson12}.
\end{remark}

An application of Theorem \ref{th:general} is the following.

\begin{lemma}\label{th:smaller}
The Berkovich Julia set $\sJ(f)$ of $f$ coincides with
\begin{gather}
 \left\{\cS\in\sP^1:
 \bigcap_{(n_j)\subset\bN:\text{infinite}}\left(\bigcap_{U:\text{open in }\sP^1,\cS\in U}\left(\bigcup_{j\in\bN}f^{n_j}(U)\right)\right)=\sP^1\setminus E(f)\right\},\label{eq:smaller}
\end{gather}
which is a priori contained in $\sJ(f)$.
\end{lemma}

\begin{proof}
 By Theorem \ref{th:inclusion}, $\sJ(f)\subset\supp\mu_f$.
 By Theorems \ref{th:general} and \ref{th:recover}, $\supp\mu_f$ is contained in \eqref{eq:smaller}.
 Clearly, \eqref{eq:smaller} is contained in $\sJ(f)$.
\end{proof}

Suppose now that $a$ is non-constant.

\begin{proof}[Proof of Theorem $\ref{th:moving}$ for archimedean $K\cong\bC$]
 We will show that $U_{\nu}\ge 0$ on $\supp\nu$. 
 Then by Lemma \ref{th:canonical}, \eqref{eq:weaklim} will hold.

 By the upper semicontinuity of $U_{\nu}$, $\{U_{\nu}<0\}$ is open. 
 Let $U$ be a component of $\{U_{\nu}<0\}$. By Lemma \ref{th:negativeFatou}, 
 $U\subset\sF(f)$. From \eqref{eq:limsup}, we have $\lim_{j\to\infty}f^{n_j}=a$
 on $U$. Since $a$ is non-constant, this implies that
 there are an $N\in\bN$ and a cyclic Fatou component $Y$ of $f$ such that $Y$ is
 a Siegel disk or an Herman ring of $f$ and that
 for every $j\ge N$, $f^{n_j}(U)\subset Y$. Then $a(U)\subset Y$.
 For some $k_0\in\bN$, we have $f^{k_0}(Y)=Y$, and for every $j\ge N$, we have $k_0|(n_j-n_N)$.

Let $h:Y\to\bC$ be a holomorphic injection (a linearization map) 
such that for some $\alpha\in\bR\setminus\bQ$, $h\circ f^{k_0}=\lambda\cdot h$
on $Y$. Taking a subsequence of $(n_j)$ if necessary,
$\lambda_0:=\lim_{j\to\infty}\lambda^{(n_j-n_N)/k_0}\in\bC$ 
exists and
\begin{gather*}
  h\circ a=\lim_{j\to\infty}h\circ f^{n_j}=\lambda_0\cdot (h\circ f^{n_N})
\end{gather*} 
 on $U$. Moreover, for every $j\in\bN$ large enough, $\lambda^{(n_j-n_N)/k_0}\neq\lambda_0$ and 
 \begin{gather*}
h\circ f^{n_j}-h\circ a
 =(\lambda^{(n_j-n_N)/k_0}-\lambda_0)\cdot (h\circ f^{n_N})\label{eq:Tortrat}
 \end{gather*}
 on $U$. Since $h$ has at most one zero in $Y$, which is simple if exists, we have
 \begin{gather*}
  0\le\nu(U)\le\limsup_{j\to\infty}\frac{1\cdot d^{n_N}}{d^{n_j}+\deg a}=0.
 \end{gather*}
 This implies that $\{U_{\nu}<0\}\cap(\supp\nu)=\emptyset$.
\end{proof}

Suppose now that $K$ is non-archimedean.
In the following definition,
$\cE_f$ is a {\itshape Berkovich version}
of Rivera-Letelier's quasiperiodicity domain of $f$.

\begin{definition}
 Let $\cE_f$ be the set of points in $\sP^1$ 
 having a neighborhood $U$ such that for some $(n_j)\subset\bN$ tending to $\infty$,
\begin{gather}
 \lim_{j\to\infty}\sup_{U\cap\bP^1}[f^{n_j},\Id_{\bP^1}]=0.\label{eq:convid} 
\end{gather}
\end{definition}

\begin{lemma}\label{th:quasiperiodicity}
$\cE_f$ is open, $f(\cE_f)\subset\cE_f$, and
$\cE_f$ is covered by singular domains of $f$. In particular,
$\cE_f\cap\bP^1\neq\bP^1$.
\end{lemma}

\begin{proof}
From the definition, $\cE_f$ is open in $\sP^1$. 
For every open subset $U$ in $\sP^1$,
$[f^{n_j},\Id]\circ f=[f^{n_j+1},f]\le L[f^{n_j},\Id]$ on $U\cap\bP^1$, 
where $L>0$ is a Lipschitz constant of $f|\bP^1$ with respect to the chordal distance.
Hence if \eqref{eq:convid} holds on the $U$, then
$\lim_{j\to\infty}\sup_{f(U)\cap\bP^1}[f^{n_j},\Id_{\bP^1}]=0$, so
$f(\cE_f)\subset\cE_f$.

By Lemma \ref{th:smaller} and \eqref{eq:convid}, $\cE_f\cap\bP^1\subset\sF(f)$.
Moreover, by \eqref{eq:convid},
$\cE_f$ is indeed covered by some cyclic Berkovich Fatou components $W$ of $f$, 
and by Theorem \ref{th:classification} and \eqref{eq:convid},
each $W$ is a singular domain.

Since $\cE_f$ is covered by singular domains of $f$, $f$ has no critical points
in $\cE_f\cap\bP^1$, so by $\deg f>1$, we have $\cE_f\cap\bP^1\neq\bP^1$. 
\end{proof}

For non-archimedean $K$ of characteristic $0$,
a non-archimedean counterpart of the uniformization 
of a Siegel disk or an Herman ring of $f$
is given by Rivera-Letelier's iterative logarithm of $f$ on $\cE_f$.

\begin{theorem}[{\cite[\S 3.2, \S4.2]{Juan03}. See also \cite[Th\'eor\`eme 2.15]{FR09}}]\label{th:uniformization}
 Suppose that $K$ has characteristic $0$ and residual characteristic $p$. 
 Let $f\in K(z)$ be a rational function on $\bP^1$
 of degree $>1$ and suppose that $\cE_f\neq\emptyset$, which implies 
 $p>0$ by \cite[Lemme 2.14]{FR09}.
 Then for every component $Y$ of $\cE_f$ not containing $\infty$,
 there are a $k_0\in\bN$, a continuous action 
 $T:\bZ_p\times (Y\cap K)\ni (\omega,y)\mapsto T^\omega(y)\in Y\cap K$
 and a non-constant $K$-valued holomorphic function $T_*$ on $Y\cap K$
 such that for every $m\in\bZ$, $(f^{k_0})^m=T^m$ on $Y\cap K$, 
 that for each $\omega\in\bZ_p$, $T^{\omega}$ is a biholomorphism on $Y\cap K$ and that
 for every $\omega_0\in\bZ_p$,
 \begin{gather}
  \lim_{\bZ_p\ni\omega\to\omega_0}\frac{T^{\omega}-T^{\omega_0}}{\omega-\omega_0}
=T_*\circ T^{\omega_0}\label{eq:logarithm}
 \end{gather} 
 locally uniformly in $Y\cap K$. 
\end{theorem}

We also need the following.

\begin{lemma}\label{th:hartogsunif}
 For every compact subset $C$ in $\{U_{\nu}<0\}$, 
\begin{gather*}
  \lim_{j\to\infty}\sup_C[f^{n_j},a]_{\can}(\cdot)=0.
\end{gather*}
\end{lemma}

\begin{proof}
 By a lemma of Hartogs (cf.\
 \cite[Proposition 2.18]{FR09}, \cite[Proposition 8.57]{BR10}) and \eqref{eq:cutoff}, 
 for every compact subset $C$ in $\sP^1$,
\begin{gather}
  \limsup_{j\to\infty}\sup_C U_{\nu_{n_j}^a}\le\sup_C U_{\nu}.\label{eq:lemmaHartogs}
\end{gather} 
By Lemma \ref{th:Riesz},
 \begin{multline*}
  \sup_C\frac{1}{d^{n_j}}\Phi(f^{n_j},a)_f(\cdot)\\
  =\sup_C U_{(1+(\deg a)/d^{n_j})\nu_{n_j}^a}+\frac{1}{d^{n_j}}\sup_C |U_{a^*\mu_f}|
  +\frac{1}{d^{n_j}}\int_{\sP^1}\Phi(f^{n_j},a)_f\rd\mu_f,
 \end{multline*}
and let us take $\limsup_{j\to\infty}$ of both sides.
Then by the comparison \eqref{eq:movingcomparison}, 
the estimate \eqref{eq:lemmaHartogs},
and the boundedness of $U_{a^*\mu_f}$, we have
 \begin{gather*}
  \limsup_{j\to\infty}\frac{1}{d^{n_j}}\log\sup_C [f^{n_j},a]_{\can}(\cdot)
  \le\sup_C U_{\nu}.
 \end{gather*} 
 If $C\subset\{U_{\nu}<0\}$, then 
 by the upper semicontinuity of $U_{\nu}$, $\sup_C U_{\nu}<0$. 
This completes the proof.
\end{proof}

Suppose now that $K$ is non-archimedean and of characteristic $0$.
By Lemma \ref{th:quasiperiodicity},
we can assume $\infty\not\in\cE_f$ without loss of generality.

\begin{proof}[Proof of Theorem $\ref{th:moving}$ for non-archimedean $K$ of characteristic $0$]
 We will show that $U_{\nu}\ge 0$ on $\supp\nu$. 
 Then by Lemma \ref{th:canonical}, \eqref{eq:weaklim} will hold.

 By the upper semicontinuity of $U_{\nu}$, $\{U_{\nu}<0\}$ is open. 
 Let $U$ be a component of $\{U_{\nu}<0\}$. 
 For every compact subset $C$ in $\{U_{\nu}<0\}$,
 $\sup_C U_{\nu}<0$.

\begin{lemma}
 $a(U)\subset\cE_f$.
\end{lemma}

\begin{proof}
 Fix $z_0\in U\cap\bP^1$. By Lemma \ref{th:hartogsunif},
 there is a Berkovich open disk $\sD$ relatively compact in $U$ and containing $z_0$ 
 such that 
 \begin{gather}
  \lim_{j\to\infty}\sup_{\sD}[f^{n_j},a]_{\can}(\cdot)=0,\label{eq:nonconstlimitrestrict}
 \end{gather} 
 and without loss of generality, 
 we can assume that $D$ is so small that $a(D)$ is a Berkovich open disc
 (see Fact \ref{th:rigid}). 
 Fix a Berkovich open disk $\sD'$ relatively compact in $a(\sD)$ and containing $a(z_0)$.
 Then by \eqref{eq:nonconstlimitrestrict},
 for every $j\in\bN$ large enough, $f^{n_j}(\sD)$ 
 is a Berkovich open disk intersecting $a(\sD)$,
 and moreover, contains $\sD'$. Hence,
 since $[f^{n_{j+1}-n_j},\Id]\circ f^{n_j}=[f^{n_{j+1}},f^{n_j}]
 \le[f^{n_{j+1}},a](\cdot)+[f^{n_j},a](\cdot)$ on $\bP^1$, we have
\begin{gather*}
 \sup_{\sD'\cap\bP^1}[f^{n_{j+1}-n_j},\Id]\le
\sup_{\sD\cap\bP^1}[f^{n_{j+1}},a](\cdot)+\sup_{\sD\cap\bP^1}[f^{n_j},a](\cdot),
\end{gather*}
 so by \eqref{eq:nonconstlimitrestrict},
 $\limsup_{j\to\infty}\sup_{\sD'\cap\bP^1}[f^{n_{j+1}-n_j},\Id]=0$.
This implies $a(U)\subset\cE_f$.
\end{proof}

 Let $Y$ be the component of $\cE_f$ containing $a(U)$.
 Let $p>0,k_0\in\bN,T,T_*$ be 
 as in Theorem \ref{th:uniformization} associated to this $Y$.

For any Berkovich closed connected affinoid $V$ in $U$,
by Lemma \ref{th:hartogsunif}, 
$\lim_{j\to\infty}\sup_V[f^{n_j},a]_{\can}(\cdot)=0$.
Then there exists
an $N\in\bN$ such that for every $j\ge N$,
the Berkovich closed connected affinoid $f^{n_j}(V)$ is 
contained in $Y$, and $k_0|(n_j-n_N)$.

For every $j\ge N$, $f^{n_j}=T^{(n_j-n_N)/k_0}\circ f^{n_N}$ 
on $V\cap\bP^1$. Taking a subsequence of $(n_j)$ if necessary,
the limit 
\begin{gather*}
  \lim_{j\to\infty}\frac{n_j-n_N}{k_0}=:\omega_0
\end{gather*} 
exists in $\bZ_p$, and
$a=\lim_{j\to\infty}f^{n_j}=\lim_{j\to\infty}T^{(n_j-n_N)/k_0}\circ f^{n_N}
 =T^{\omega_0}\circ f^{n_N}$
on $V\cap\bP^1$.
For every $j\ge N$, 
\begin{gather}
 f^{n_j}-a=
(T^{(n_j-n_N)/k_0}-T^{\omega_0})\circ f^{n_N}
\label{eq:zeros}
\end{gather}
on $V\cap\bP^1$, and increasing $N$ if necessary,
we also have $(n_j-n_N)/k_0\neq\omega_0$.

Let $Z_*$ be the set of all zeros in the closed connected affinoid
$f^{n_N}(V)\cap K$
of the non-constant holomorphic function $T_*\circ T^{\omega_0}$ on $Y\cap K$.
Then $\#Z_*<\infty$ (see Fact \ref{th:holo}). Hence $\#f^{-n_N}(Z_*)<\infty$,
and we can assume that $f^{-n_N}(Z_*)\subset K$ without loss of generality. 

Now we also assume that the Berkovich closed
connected affinoid $V$ is strict.

\begin{lemma}\label{th:atomic}
 $(\supp\nu)\cap ((\Int V)\setminus f^{-n_N}(Z_*))=\emptyset$.
\end{lemma}

\begin{proof}
 For each $\epsilon>0$ in $|K^*|$, set
\begin{gather*}
 V_{\epsilon}:=V\setminus \bigcup_{w\in f^{-n_N}(Z_*)}\{\cS\in\sP^1\setminus\{\infty\}:|\cS-w|<\epsilon\},
\end{gather*}
 which is a strict Berkovich closed connected affinoid.
 Then $f^{n_N}(V_{\epsilon})$ is a strict
 Berkovich closed connected affinoid in $Y$.
 Hence by the maximum modulus principle, the minimum
 \begin{gather*}
  \min\{|T_*\circ T^{\omega_0}(z)|:z\in f^{n_N}(V_{\epsilon})\cap K\}>0
 \end{gather*}
 exists (see Fact \ref{th:rigid}) and is positive by the choice of $V_{\epsilon}$.
 Then from the uniform convergence \eqref{eq:logarithm}
 on $f^{n_N}(V_{\epsilon})\cap K$, 
 for every $j\in\bN$ large enough, 
\begin{gather*}
 |T^{(n_j-n_N)/k_0}-T^{\omega_0}|>0 
\end{gather*} 
 on $f^{n_N}(V_{\epsilon})\cap K$, which with \eqref{eq:zeros} implies that
 there is no root of $f^{n_j}=a$ in $V_{\epsilon}\cap\bP^1$. 
 Hence $(\supp\nu)\cap\Int V_{\epsilon}=\emptyset$, which implies that 
 $(\supp\nu)\cap((\Int V)\setminus f^{-n_N}(Z_*))=\emptyset$.
\end{proof}

\begin{lemma}\label{th:complement}
 $(\supp\nu)\cap((\Int V)\cap f^{-n_N}(Z_*))=\emptyset$.
\end{lemma}

\begin{proof}
 Let $z_0\in(\Int V)\cap f^{-n_N}(Z_*)$.
 If $z_0$ is a root of $f^{n_j}=a$, then by
 \eqref{eq:zeros} and the uniform convergence \eqref{eq:logarithm} on $V$,
 the multiplicity of $z_0$
 as a root of $f^{n_j}=a$ is bounded from above by
 \begin{gather}
 (\deg_{f^{n_N}(z_0)}(T_*\circ T^{\omega_0}))\cdot d^{n_N}-1.\label{eq:mult}
 \end{gather}
 For any Berkovich open disk $\sD$ in $V$ containing $z_0$ and satisfying
 $\overline{\sD}\cap f^{-n_N}(Z_*)=\{z_0\}$, 
 from the upper bound \eqref{eq:mult} and Lemma \ref{th:atomic},
 \begin{align*}
 0\le&\limsup_{j\to\infty}\nu_{n_j}^a(\sD)
 \le\limsup_{j\to\infty}\nu_{n_j}^a(\{z_0\})+\limsup_{j\to\infty}\nu_{n_j}^a(\sD\setminus\{z_0\})\\
 \le&\limsup_{j\to\infty}\frac{(\deg_{f^{n_N}(z_0)}(T_*\circ T^{\omega_0}))\cdot d^{n_N}}{d^{n_j}}
  +\nu((\Int V)\setminus f^{-n_N}(Z_*))=0.
 \end{align*}
 Hence $\nu(\sD)=0$ if $\sD$ is small enough, so
 $z_0\not\in\supp\nu$.
\end{proof}

From Lemmas \ref{th:atomic} and \ref{th:complement}, $(\Int V)\cap(\supp\nu)=\emptyset$.
This implies that $U\cap(\supp\nu)=\emptyset$, so $\{U_{\nu}<0\}\cap(\supp\nu)=\emptyset$.
\end{proof}

Now the proof of Theorem \ref{th:moving} is complete.

\section{The case of polynomials}\label{sec:example}

Let $K$ be an algebraically closed field of any characteristic and complete with respect
to a non-trivial absolute value.

For every polynomial $\phi\in K[z]$ on $\bP^1$, the factorization of $\phi$ extends
$|\phi|$ continuously to $\sP^1\setminus\{\infty\}$ using the extended
$|\cdot-w|$ on $\sP^1\setminus\{\infty\}$ for each $w\in\bP^1\setminus\{\infty\}$. 
For polynomials $\phi_i\in K[z]$ ($i\in\{1,2\}$), $\phi_1-\phi_2$ is also a polynomial.
Hence the continuous extension $\cS\mapsto|\phi_1-\phi_2|_{\can}(\cS)$ to $\sP^1\setminus\{\infty\}$
of the function $z\mapsto|\phi_1(z)-\phi_2(z)|$ on $\bP^1\setminus\{\infty\}$
exists so that on $\sP^1\setminus\{\infty\}$,
\begin{gather}
  |\phi_1-\phi_2|_{\can}(\cdot)=[\phi_1,\phi_2]_{\can}(\cdot)\max\{1,|\phi_1(\cdot)|\}\max\{1,|\phi_2(\cdot)|\}.\label{eq:polyext}
\end{gather} 

Let $f\in K[z]$ is a polynomial on $\bP^1$ of degree $d>1$. 
The Berkovich filled-in Julia set of $f$ is
\begin{gather*}
 \sK(f):=\{\cS\in\sP^1:\lim_{n\to\infty}f^n(\cS)\neq\infty\}.
\end{gather*} 
 Noting that $f(\infty)=\infty\in E(f)$,
 let $\sA_{\infty}=\sA_{\infty}(f)$ be the fixed
 immediate attractive basin of $f$ containing $\infty$.
 Then $f^{-1}(\sA_{\infty})=\sA_{\infty}$
 since $\deg(f:\sA_{\infty}\to\sA_{\infty})=\deg_{\infty}f=d$.
 Hence $\sA_{\infty}$ is completely invariant under $f$, and
 $\sK(f)=\sP^1\setminus\sA_{\infty}$.
 Moreover, $\partial\sA_{\infty}=\partial\sK(f)=\sJ(f)$. Indeed,
 by Theorem \ref{th:inclusion}, $\sJ(f)\subset\supp\mu_f$.
 Fix $\cS\in\sA_{\infty}\cap\bP^1$. Then 
 by Theorem \ref{th:general}, $\supp\mu_f\subset\bigcap_{N\in\bN}
\overline{\bigcup_{n\ge N}f^{-n}(\cS)}$, which is contained in
$\partial\sA_{\infty}\subset\sJ(f)$.

For each $R>0$ in $|K^*|$, let $\sD_R^*:=\{\cS\in\sP^1\setminus\{\infty\}:|\cS|>R\}$
and $\sD_R:=\sD_R^*\cup\{\infty\}$. If $R>0$ is large enough, then
since $\infty$ is a (super)attracting fixed point of $f$, 
we have $\inf_{z\in\sD_R^*\cap\bP^1}|f(z)|>R$. Hence by the
continuity of $|f(\cdot)|$, $\inf_{\sD_R^*}|f(\cdot)|>R$. This implies that
$\sD_R\Subset f^{-1}(\sD_R)$. Since $\sA_{\infty}=\bigcup_{n\in\bN}f^{-n}(\sD_R)$,
for every Berkovich closed disk $\sD$ in $\sA_{\infty}\setminus\{\infty\}$, 
we have $\liminf_{n\to\infty}\inf_{\sD}|f^n(\cdot)|>R$. Hence we have
\begin{gather}
 \liminf_{n\to\infty}\inf_{\sD}|f^n(\cdot)|=\infty.\label{eq:diverge}
\end{gather}

\begin{lemma}\label{th:recurrence}
 Suppose that $K$ is non-archimedean. For every 
 polynomial $f\in K[z]$ on $\bP^1$ of degree $d>1$ and 
 every polynomial $a\in K[z]$ on $\bP^1$,
 the condition \eqref{eq:hyperbolic} holds 
 on $\sA_{\infty}(f)\setminus\{\infty\}$, and on $\sK(f)$,
 $\sup_{n\in\bN}|\log[f^n,a]_{\can}(\cdot)-\log|f^n-a|_{\can}(\cdot)|<\infty$.

In particular, \eqref{eq:hyperbolic} holds on $\sP^1\setminus\bP^1$
if and only if the condition
\begin{gather}
 \lim_{j\to\infty}\frac{1}{d^{n_j}}\log|f^{n_j}-a|_{\can}(\cdot)=0\tag{\ref{eq:hyperbolic}'}
 \label{eq:hyperbolicpoly}
\end{gather}
holds on $\sK(f)\setminus\bP^1$.
\end{lemma}

\begin{proof}
For every Berkovich closed disk $\sD$ in $\sA_{\infty}\setminus\{\infty\}$,
fix an $R>0$ in $|K^*|$ so large that $R>\max\{1,\sup_{\sD}|a(\cdot)|\}$. 
By \eqref{eq:polyext} and \eqref{eq:diverge}, for every $n\in\bN$ large enough, on $\sD\cap\bP^1$, 
$\log[f^n(\cdot),a(\cdot)]=\log|f^n(\cdot)|-\log|f^n(\cdot)|-\log\max\{1,|a(\cdot)|\}\ge -\log R$.
Hence $\log[f^n,a]_{\can}(\cdot)\ge -\log R$ on $\sD$
since both sides are continuous. This implies that 
 \eqref{eq:hyperbolic} holds 
 on $\sA_{\infty}(f)\setminus\{\infty\}$.

Next, fix an $R>0$ in $|K^*|$ so large that $\sD_R\subset\sA_{\infty}$. Then
$\bigcup_{n\in\bN}f^n(\sK(f))\subset\sP^1\setminus\sD_R$.
Hence by \eqref{eq:polyext}, 
\begin{multline*}
 \sup_{n\in\bN}|\log[f^n,a]_{\can}(\cdot)-\log|f^n-a|_{\can}(\cdot)|\\
 \le\log\max\{1,R\}+\log\max\{1,|a(\cdot)|\}<\infty
\end{multline*}
on $\sK(f)$. 
\end{proof}

We conclude this section with an example.
Suppose that $K$ has characteristic $p>0$, and
set $f(z)=z+z^p$ and $a=\Id$. 
Then $\sK(f)=\{\cS\in\sP^1\setminus\{\infty\}:|\cS|\le 1\}$. 
For each $j\in\bN$,
$f^{p^j}(z)=z+z^{p^{p^j}}$, 
and the equality
\begin{gather*}
 \log|f^{p^j}-\Id|_{\can}=p^{p^j}\log|\cdot|
\end{gather*}
holds on $\bP^1\setminus\{\infty\}$.
By the continuity of both sides, this extends to $\sP^1\setminus\{\infty\}$.
In particular, \eqref{eq:hyperbolicpoly} does not hold on $\sK(f)\setminus\bP^1$. 
Hence the equidistribution property \eqref{eq:general} 
for $f(z)=z+z^p$ and $a=\Id$ does not hold.

Of course, this could be more directly seen since
\begin{gather*}
 \lim_{j\to\infty}\nu_{p^j}^a=\lim_{j\to\infty}\frac{1}{p^{p^j}+1}(p^{p^j}\delta_0+\delta_{\infty})=\delta_0
\end{gather*}
weakly on $\sP^1$, but $\supp\mu_f=\sJ(f)=\partial\sK(f)=\{\cS_{\can}\}$.

\subsection*{Acknowledgements}
The author thanks Professors Charles Favre and 
Juan Rivera-Letelier for invaluable comments,
Professors Mattias Jonsson and William Gignac for 
stimulating discussion on the Problem,
and the referee for very careful scrutiny and pertinent comments.
This research was partially supported by JSPS Grant-in-Aid for Young Scientists (B), 21740096.

\def\cprime{$'$}

\end{document}